\documentclass[12pt, reqno]{amsart}

\usepackage{amsmath,amssymb,amsthm,amsfonts,verbatim}
\usepackage{microtype}
\usepackage[all,2cell]{xy}
\usepackage{mathtools}
\usepackage{graphicx}
\usepackage{pinlabel}
\usepackage{hyperref}
\usepackage{mathrsfs}
\usepackage{color}
\usepackage[dvipsnames]{xcolor}
\usepackage{enumerate}
\usepackage{cite}
\usepackage{soul}
\usepackage{palatino}
\usepackage{wrapfig}

\CompileMatrices

\usepackage[top=1.2in,bottom=1.2in,left=1in,right=1in]{geometry}

\usepackage{hyperref} 
\hypersetup{          
	colorlinks=true, breaklinks, linkcolor=[RGB]{51 102 204}, filecolor=Orchid, urlcolor=[RGB]{51 102 204},
	citecolor=Orchid, linktoc=all, }
\usepackage[nameinlink]{cleveref}

\theoremstyle{plain}
\newtheorem{theorem}{Theorem}[section]

\newtheorem{proposition}[theorem]{Proposition}
\newtheorem{lemma}[theorem]{Lemma}

\newtheorem{corollary}[theorem]{Corollary}

\theoremstyle{definition}
\newtheorem{definition}[theorem]{Definition}
\newtheorem{example}[theorem]{Example}

\newtheorem{remark}[theorem]{Remark}

\newcommand{\nc}{\newcommand}
\nc{\dmo}{\DeclareMathOperator}

\nc{\Q}{\mathbb{Q}}
\nc{\F}{\mathbb{F}}
\nc{\R}{\mathbb{R}}
\nc{\Z}{\mathbb{Z}}
\nc{\C}{\mathbb{C}}
\nc{\cN}{\mathcal{N}}
\nc{\N}{\mathbb{N}}
\nc{\Ell}{\mathcal{L}}
\nc{\cM}{\mathcal{M}}
\nc{\K}{\mathcal{K}}
\nc{\I}{\mathcal{I}}
\nc{\cP}{\mathcal{P}}
\nc{\cS}{\mathcal{S}}
\nc{\cT}{\mathcal T}
\nc{\U}{\mathcal U}
\nc{\disk}{\mathbb{D}}
\nc{\hyp}{\mathbb{H}}

\nc{\CP}{\mathbb{CP}}
\nc{\RP}{\mathbb{RP}}
\dmo{\Mod}{Mod}
\dmo{\PMod}{PMod}
\dmo{\LMod}{LMod}
\dmo{\Diff}{Diff}
\dmo{\Homeo}{Homeo}
\dmo{\dist}{dist}
\dmo\BDiff{BDiff}
\dmo\SO{SO}
\dmo\Hom{Hom}
\dmo\SL{SL}
\dmo\rank{rank}
\dmo\sig{sig}
\dmo\Out{Out}
\dmo\Aut{Aut}
\dmo\Inn{Inn}
\dmo\GL{GL}
\dmo\PGL{PGL}
\dmo\Gr{Gr}
\dmo\PSL{PSL}
\dmo\BHomeo{BHomeo}
\dmo\EHomeo{EHomeo}
\dmo\EDiff{EDiff}
\dmo\Disc{Disc}
\dmo\Aff{Aff}

\dmo\Teich{Teich}
\dmo\Fix{Fix}
\nc{\pair}[1]{\ensuremath{\left\langle #1 \right\rangle}}
\nc{\abs}[1]{\ensuremath{\left| #1 \right|}}
\nc{\action}{\circlearrowright}
\nc{\norm}[1]{\left | \left | #1 \right | \right |}
\nc{\abcd}[4]{\ensuremath{\left(\begin{array}{cc} #1 & #2 \\ #3 & #4 \end{array}\right)}}
\nc{\into}\hookrightarrow
\dmo{\Isom}{Isom}
\nc{\normal}{\vartriangleleft}
\dmo{\Vol}{Vol}
\dmo{\im}{Im}
\dmo{\Push}{Push}
\dmo{\Conf}{Conf}
\dmo{\PConf}{PConf}
\dmo{\PB}{PB}
\dmo{\id}{id}
\dmo{\Jac}{Jac}
\dmo{\Pic}{Pic}
\dmo{\Stab}{Stab}
\dmo{\Arf}{Arf}
\dmo{\End}{End}
\dmo{\Gal}{Gal}
\dmo{\lcm}{lcm}
\dmo{\ab}{ab}
\dmo{\opp}{op}
\dmo{\SU}{SU}
\dmo{\OT}{\Omega \mathcal{T}}
\dmo{\OM}{\Omega \mathcal{M}}
\dmo{\PH}{\mathbb{P}\mathcal{H}}
\dmo{\spin}{spin}
\dmo{\even}{even}
\dmo{\odd}{odd}
\dmo{\comp}{\mathcal{H}}
\dmo{\Mgk}{\mathcal{M}_{g, \underline{\kappa}}}
\dmo{\orb}{orb}
\dmo{\AJ}{AJ}
\dmo{\Ck}{\mathsf{C}(\underline{\kappa})}
\dmo{\Int}{Int}
\dmo{\pr}{pr}
\dmo{\lab}{lab}
\dmo{\Sym}{Sym}
\dmo{\Ann}{Ann}
\dmo{\Rad}{Rad}
\dmo{\Ind}{Ind}
\dmo{\Div}{Div}
\dmo{\Res}{Res}
\dmo{\Hur}{Hur}
\dmo{\vcd}{vcd}

\nc{\Span}[1]{\operatorname{Span}(#1)}

\renewcommand{\epsilon}{\varepsilon}

\renewcommand{\le}{\leqslant}
\nc{\coloneq}{\mathrel{\mathop:}\mkern-1.2mu=}
\nc{\margin}[1]{\marginpar{\scriptsize #1}}
\nc{\para}[1]{\medskip\noindent\textbf{#1.}}
\definecolor{myblue}{RGB}{102,153, 255}
\definecolor{myred}{RGB}{204,0,0}
\definecolor{mygreen}{RGB}{0,204,0}
\definecolor{myorange}{RGB}{255,102,0}
\definecolor{mypurple}{RGB}{138,43,226}
\nc{\red}[1]{\textcolor{myred}{#1}}
\nc{\blue}[1]{\textcolor{myblue}{#1}}

\nc{\Sk}{\mathcal S_{\kappa}}
\nc{\Pk}{\mathcal P_{\kappa}}

\nc{\lb}{[}
\nc{\rb}{]}

\title{Ropes, fractions, and moduli spaces}

\author{Nick Salter}
\email{nsalter@nd.edu}
\thanks{The author is supported by NSF Award No. DMS-2153879.}
\address{Department of Mathematics, University of Notre Dame, Hurley Hall, Notre Dame, IN 46556}
\date{September 22, 2023}

\begin{document}
\maketitle

\begin{abstract}This is an exposition of John H. Conway's {\em tangle trick}. We discuss {\em what} the trick is, {\em how} to perform it, {\em why} it works mathematically, and finally offer a conceptual explanation for why a trick like this should exist in the first place. The mathematical centerpiece is the relationship between braids on three strands and elliptic curves, and we a draw a line from the tangle trick back to work of Weierstrass, Abel, and Jacobi in the 19th century. For the most part we assume only a familiarity with the language of group actions, but some prior exposure to the fundamental group is beneficial in places.
\end{abstract}

\section{What is the tangle trick?} I learned of the tangle trick from my graduate school office mate Tim Black, who learned it from Conway himself at the Canada-USA Mathcamp. Here is what you would see watching a performance of the trick.

\begin{wrapfigure}{l}{0.18\textwidth}
\centering
\labellist
\small
\pinlabel $1$ at 12 12
\pinlabel $2$ at 103 12
\pinlabel $3$ at 103 103
\pinlabel $4$ at 12 103
\endlabellist
\includegraphics[scale = 0.7]{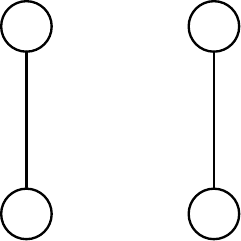}
\end{wrapfigure} 
A {\sc magician} and their {\sc assistant} take the stage with two ropes of equal length. They recruit five {\sc volunteers} from the audience. Four of the {\sc volunteers} they position at the corners of a square, handing them the ropes as shown at left. The fifth {\sc volunteer} is the {\sc caller}, and is positioned off to the side. The {\sc caller} is tasked with calling out one of two {\sc moves} for the {\sc volunteers} to perform: {\bf T}wist and tu{\bf R}n.

To perform a {\bf T}wist, the {\sc volunteers} standing in positions 1 and 2 exchange their ropes, with 1 passing their end {\em over}, and 2 passing their rope {\em under}. To perform a tu{\bf R}n, the four {\sc volunteers} rotate positions counterclockwise along the vertices of their square.
\begin{figure}[htbp]
\centering
\labellist
\small
\pinlabel $T$ at 155 65
\pinlabel $b$ at 58 60
\pinlabel $b$ at 254 60
\endlabellist
\includegraphics[scale = 0.7]{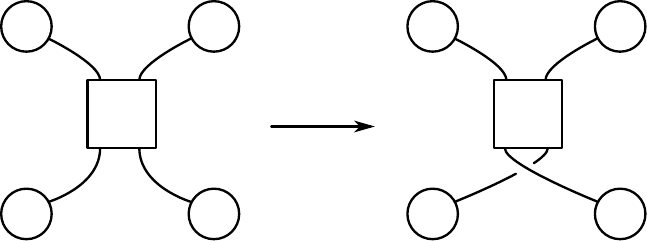}
\caption{The {\bf T}wist move. The region enclosed in the square and labeled ``b'' indicates an arbitrary tangle contained within.}
\label{figure:twist}
\end{figure}

\begin{figure}[htbp]
\centering
\labellist
\small
\pinlabel $R$ at 155 65
\pinlabel $b$ at 58 60
\pinlabel \rotatebox[origin=c]{90}{$b$} at 252 58
\endlabellist
\includegraphics[scale = 0.7]{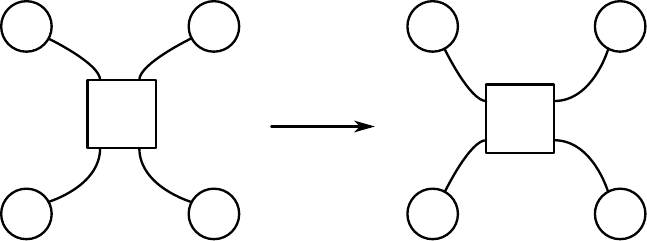}
\caption{The tu{\bf R}n move.}
\label{figure:turn}
\end{figure}

The trick now proceeds as follows. The {\sc magician} exits the room after instructing the {\sc caller} to call out a sequence of {\bf T}wists and tu{\bf R}ns of their choosing, until the {\sc assistant} says to stop. As the moves are performed, the pair of ropes becomes increasingly tangled. When the {\sc assistant} calls stop, they {\em wrap} the tangled middle portion of the ropes in a bag. The {\sc magician} is then summoned back into the room, and the {\sc assistant} tells them {\em a single rational number}.

Now the {\sc magician} takes over, {\em calling out a further sequence of {\bf T}wists and tu{\bf R}ns} (in the same directions as before!). Finally, the big reveal: the {\sc magician} removes the bag, and after a good shake, {\em the ropes have become untangled}. Magic!

\para{Variants} It seems that Conway performed the trick in various ways. Curt McMullen reports on a performance at Berkeley, where a pre-constructed tangle already wrapped in a bag appeared on stage along with its rational number. Conway then solicited untangling suggestions from the audience, updating the fraction until the invariant reached zero, at which point the bag was removed and the tangle was resolved.

\section{How do you perform the trick?}
Let us abbreviate a {\bf T}wist by T and a tu{\bf R}n by R. The method is to associate to each state of the ropes a {\em rational number} in such a way that each T or R alters this number in a predictable fashion. In fact, this ``tangle invariant'' will take values in $\Q \cup \{\infty\}$. 

Assign the initial, uncrossed state the invariant $x = 0$. Subsequently, apply the following rules: a T sends $x$ to $x + 1$, and an $R$ sends $x$ to $-1/x$. We adopt the following conventions for arithmetic with $\infty$: we define $-1/0 = \infty$, $-1/\infty = 0$, and $\infty + 1 = \infty$. 

As the {\sc caller} is calling their sequence of T's and R's, the {\sc assistant} is silently tracking the value of the tangle invariant (this can be very challenging in practice!). The number they present to the {\sc magician} upon their return is the invariant of the present state. 

To untangle the ropes, the {\sc magician} is calling a sequence of T's and R's so as to take the tangle invariant back to zero. The general method is to {\em reduce the denominator of $x$} by a sequence of moves $T^kR$ (written so as to act on $x$ from the left, so that the $R$ is performed {\em first}). This amounts to the Euclidean algorithm, as we can see in the following example.

\begin{example}\label{example}
Suppose the {\sc magician} returns to the room and is presented with the number $146/57$. The untangling begins as shown:
\[
\frac{146}{57} \xrightarrow{R} \frac{-57}{146} \xrightarrow{T} \frac{89}{146} \xrightarrow{R} \frac{-146}{89} \xrightarrow{T} \frac{-57}{89}.
\]
Thus from the second state to the fifth, the fraction reduces from $\frac{-57}{146}$ to $\frac{-57}{89}$. While the fraction {\em originally} had a smaller denominator of $57$, this was something of a red herring, since a positive value of $x$ will have to be resolved by first performing $R$, so that $T$ will then {\em decrease} the absolute value of the numerator. We continue:
\[
\frac{-57}{89} \xrightarrow{T} \frac{32}{89} \xrightarrow{R} \frac{-89}{32} \xrightarrow{T} \frac{-57}{32} \xrightarrow{T} \frac{-25}{32} \xrightarrow{T} \frac{7}{32}.
\]
At this point, the methodology is hopefully clear: given a fraction $\frac{p}{q}$ with $0 < p < q$, an $R$ will convert this to $\frac{-q}{p}$. Then apply $k$ $T$'s, for $k$ such that $0 \le kp - q < p$, and repeat. In our example, we finish as follows:
\[
\frac{7}{32} \xrightarrow{R} \frac{-32}{7} \xrightarrow{T^5} \frac{3}{7} \xrightarrow{R} \frac{-7}{3} \xrightarrow{T^3} \frac{2}{3} \xrightarrow{R} \frac{-3}{2} \xrightarrow{T^2} \frac{1}{2} \xrightarrow{R} -2 \xrightarrow{T^2} 0.
\]
\end{example}
 
\section{Why does the trick work?}
Perhaps you've noticed the crucial unstated assumption underlying the method above: {\em any tangle with invariant $0$ must be the initial untangled state}. The purpose of this section is to explain why this is the case. Let us introduce some terminology: a {\em rational tangle} is any ``legal'' state of the ropes, i.e. any configuration obtained by starting with the {\em untangle} (the initial configuration) and performing $T$ and $R$ moves. This notion was introduced in Conway's paper \cite{conway}. A rational tangle has a {\em tangle invariant} valued in $\Q\cup \{\infty\}$ which transforms as discussed above under the moves $T$ and $R$. The claim that the trick works then amounts to the following statement:

\begin{proposition}\label{untangle0}
    If a rational tangle has invariant $0$, then it is the untangle.
\end{proposition}

\begin{remark}\label{remark:notWD}
    There is a crucial subtlety about the tangle invariant: {\em we do not know it is well-defined}. That is, it is not clear that if one starts with the untangle and then performs two separate sets of moves, obtaining the same tangle both ways, that the invariants assigned to the end results agree. The tangle invariant as we have introduced it here is in fact an invariant of the {\em sequence of moves} used to produce the current state. {\em Nevertheless}, it is possible to prove that the trick works based just on this weaker fact! 

    It would take us a bit far afield from the main axis of our discussion to prove the well-definedness of the tangle invariant. A proof based on the Jones polynomial appears in \cite{kauffman}.
\end{remark}

The proof will be based around a study of group actions. Ultimately we will consider {\em three} group actions, but \Cref{untangle0} will only require a discussion of the first two.

\para{Group action 1: twists and turns} Let us define the group $\Gamma$ to be the free group on the letters $R$ and $T$. Then $\Gamma$ acts on the set $\cT$ of rational tangles via the procedure discussed above and shown in \Cref{figure:twist,figure:turn}. The action of $\Gamma$ on $\cT$ is certainly not {\em faithful}\footnote{Recall that an action of $G$ on $S$ is {\em faithful} if every nonidentity element $g \in G$ acts nontrivially. Equivalently, under the description of the group action as a homomorphism $G \to \Aut(S)$ (where $\Aut(S)$ denotes the group of permutations of $S$), a faithful action is one for which this map is {\em injective}.}, since, e.g. $R^4$ obviously acts trivially. There are in fact more such ``relations'' (combinations of $R, T$ that leave all rational tangles invariant). Understanding what these are will be a central question going forward - see \Cref{remark:matricesact}.

\para{Group action 2: M\"obius transformations} The group $\SL_2(\Z)$ of $2\times 2$ matrices with integer entries and unit determinant acts on the set $\Q\cup \{\infty\}$ by so-called {\em M\"obius transformations} as follows:
\[
\begin{pmatrix}
a & b \\ c & d
\end{pmatrix}
\cdot x = \frac{ax + b}{cx + d}.
\]
In fact, since $-I$ acts trivially, this descends to an action of $\PSL_2(\Z) := \SL_2(\Z) / \pair{-I}$, which is easily seen to be faithful.

In this action, the transformation $x \mapsto x + 1$ then corresponds to the action of the matrix $t=\begin{pmatrix}1&1\\0&1\end{pmatrix}$, and the transformation $x \mapsto \frac{-1}{x}$ corresponds to $r = \begin{pmatrix}
    0 & -1\\1&0
\end{pmatrix}$. 
Since $\Gamma$ (as defined in Group Action 1) is {\em free}, there is a homomorphism $\phi: \Gamma \to \PSL_2(\Z)$ given by sending $R$ to $r$ and $T$ to $t$, and hence we can regard this second action as likewise being by the group $\Gamma$. While essentially trivial, this remark will turn out to be structurally illuminating, since it will allow us to regard each of the sets $\cT$ and $\Q\cup\{\infty\}$ as carrying actions of the same group $\Gamma$.

\subsection{A little about $\PSL_2(\Z)$}
It turns out that $\phi: \Gamma \to \PSL_2(\Z)$ is a {\em surjection}, and in fact, there is a simple presentation for $\PSL_2(\Z)$ with respect to this generating set which we will use in the ensuing discussion.

\begin{theorem}\label{theorem:psl2z}
    $\PSL_2(\Z)$ is generated by $r$ and $t$, and with respect to this generating set, there is a presentation
    \[
    \PSL_2(\Z) \cong \pair{r,t \mid r^2 = (tr)^3 = id}.
    \]
\end{theorem}
    \begin{proof}
        Conceptually, one can analyze the action of $\PSL_2(\Z)$ on the {\em upper half-plane} by the same M\"obius action as in Group Action 2, using the principle that a presentation can be extracted from information encoded in a fundamental domain (we will discuss this action further in \Cref{section:EC}).  A short and slick proof is given in \cite{alperin}. 
    \end{proof}

\subsection{Symmetries} Rational tangles have some symmetries that will be essential to understand for the proof of \Cref{untangle0}.

\begin{definition}[$bdpq$ symmetry]
    A tangle $\tau$ has {\em $bdpq$ symmetry} if it is invariant under rotation through horizontal and vertical axes as shown in \Cref{figure:bdpq}. 
\end{definition}

\begin{figure}[htbp]
\centering
\labellist
\tiny
\pinlabel $b$ at 40 163
\pinlabel $b$ at 40 42
\pinlabel $d$ at 209 163
\pinlabel $p$ at 198 49
\endlabellist
\includegraphics[scale = 0.7]{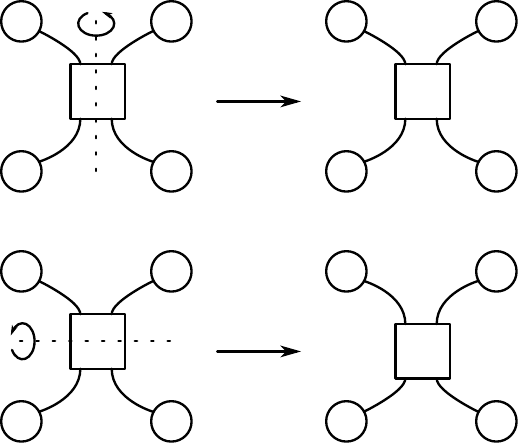}
\caption{$bdpq$ symmetry.}
\label{figure:bdpq}
\end{figure}
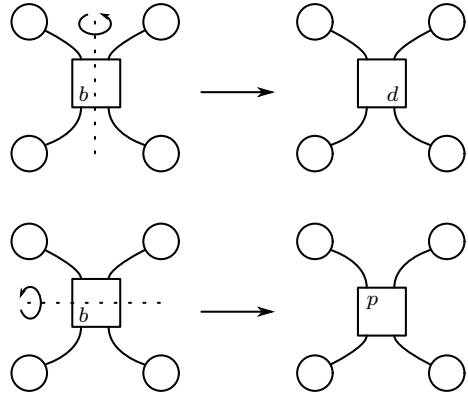

Of course, the terminology ``$bdpq$ symmetry'' is ad hoc; it is more conventional to discuss ``$\Z/2 \times \Z/2$ symmetry'', or, somewhat more classically, ``Klein 4-group symmetry''. But I feel that the term $bdpq$ symmetry is usefully evocative in this setting. 

\begin{lemma}\label{lemma:symmetric}
    Every rational tangle has $bdpq$ symmetry.
\end{lemma}
\begin{proof}
    We proceed by induction on the number $n$ of $T$ and $R$ moves used to take the untangle $\tau_0$ to $\tau$. For the base case $n = 0$, note that $\tau_0$ certainly has $bdpq$ symmetry. 

    Now suppose that $\tau$ is a rational tangle with $bdpq$ symmetry; we seek to show that $T \cdot \tau$ and $R \cdot \tau$ also have $bdpq$ symmetry. 

\begin{figure}[htbp]
\centering
\labellist
\pinlabel $\sim$ at 260 165
\pinlabel $\sim$ at 260 45
\pinlabel $\sim$ at 382 45
\tiny
\pinlabel $b$ at 40 163
\pinlabel $b$ at 40 42
\pinlabel $b$ at 315 163
\pinlabel $d$ at 209 163
\pinlabel $p$ at 198 49
\pinlabel $q$ at 329 55
\pinlabel $b$ at 435 42
\endlabellist
\includegraphics[scale = 0.7]{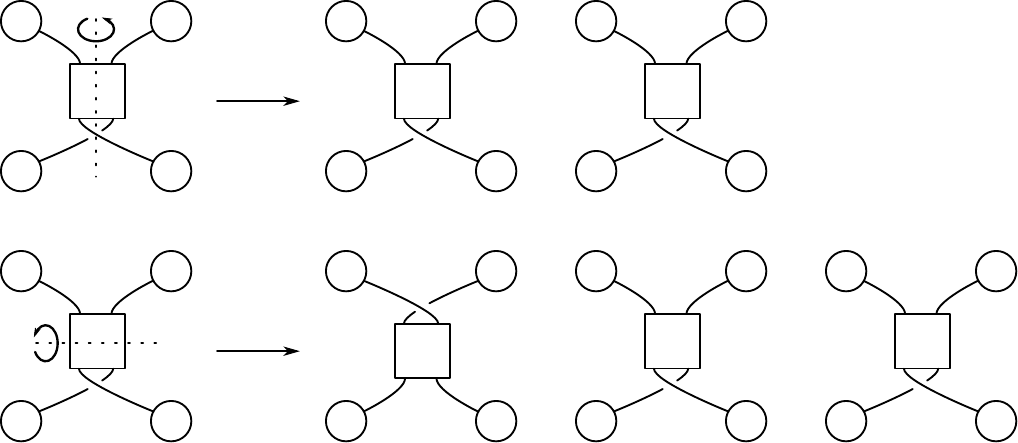}
\caption{Verifying that $T$ preserves $bdpq$ symmetry. Along the bottom, we have performed a ``flype'' move and then exploited $bdpq$ symmetry of the subtangle.}
\label{figure:tsym}
\end{figure}
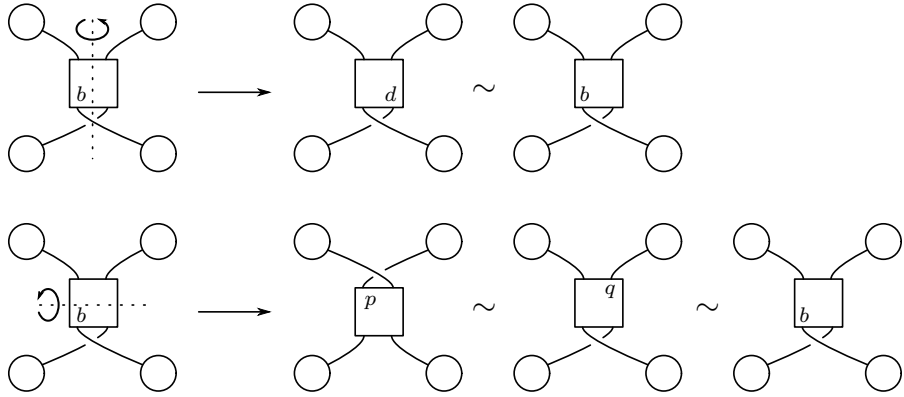

    The argument for $T \cdot \tau$ is carried out in \Cref{figure:tsym}. The verification for $R \cdot \tau$ is straightforward. Either one checks pictorially as we did for $T$, or else appeals to the following principle: suppose a group $G$ acts on a set $S$, and $H \normal G$ is a normal subgroup. Let $s \in S$ be fixed by the action of $H$. Then any $g\cdot s$ is also a fixed point for the action of $H$. This can be applied to the action of $G = D_8$ the dihedral group of order $8$ (symmetries of a square) acting on the set of tangles, with $H \normal G$ the $bdpq$ symmetry group (this is of index $2$, hence normal). The nontrivial coset of $G/H$ is represented by the action of $R$, proving the result. 
    \end{proof}

\subsection{Proof of \Cref{untangle0}}
    Let $\tau \in \cT$ be a rational tangle with invariant zero. As $\tau$ is a rational tangle, it is obtained from the untangle $\tau_0$ by a sequence of twists and turns under Group Action 1 of $\Gamma$ on $\cT$ described by some word $W$ in the generators $R, T$ of $\Gamma$:
    \[
    \tau = W \cdot \tau_0.
    \]
    In light of \Cref{remark:notWD}, the condition that $\tau$ has invariant $0$ precisely means that the corresponding word $W$, acting via Group Action 2 on $\Q\cup\{\infty\}$, satisfies
    \[
    W \cdot 0 = 0,
    \]
    or more structurally, that $W$ is contained in the {\em stabilizer subgroup} of $0$: 
    \[
    W \in \Stab^2_{\Gamma}(0) \le \Gamma,
    \]
    where the superscript $2$ indicates that this is a stabilizer for Group Action 2. We seek to show that $W$ is then necessarily contained in the stabilizer of $\tau_0$:
    \[
    \Stab_\Gamma^2(0) \le \Stab^1_{\Gamma}(\tau_0).
    \]
    (As an incidental remark, observe that the assertion that the tangle invariant is well-defined amounts to showing the opposite containment.)

    To prove this, we will analyze the structure of $\Stab_\Gamma^2(0)$. As a first observation, since Group Action $2$ factors through the surjection $\phi: \Gamma \to \PSL_2(\Z)$, there is an equality
    \[
    \Stab_\Gamma^2(0) = \phi^{-1}(\Stab_{\PSL_2(\Z)}(0)).
    \]
    Thus, $\Stab_\Gamma^2(0)$ is generated by two types of elements: (a) lifts of generators of $\Stab_{\PSL_2(\Z)}(0)$, and (b) generators of $\ker(\phi)$. We will analyze each in turn.

    Suppose then that $\pm \begin{pmatrix} a&b\\c&d\end{pmatrix} \in \PSL_2(\Z)$ stabilizes $0$. Applying the formula for the group action,
    \[
    \begin{pmatrix} a & b\\ c&d \end{pmatrix} \cdot 0 = \frac{b}{d},
    \]
    showing that $b = 0$. The determinant condition then implies that, up to sign, $a = d = 1$. This gives the following description:
    \[
    \Stab_{\PSL_2(\Z)}(0) = \pair{\begin{pmatrix} 1 & 0 \\ 1 & 1 \end{pmatrix}}.
    \]
    One verifies the factorization $\begin{pmatrix} 1 & 0 \\ -1 & 1 \end{pmatrix} = trt$. Thus, we need to verify that $TRT \in \Gamma$ fixes the untangle $\tau_0$ under Group Action 1. This is easily checked with a quick diagrammatic calculation: the tangle $RT \cdot \tau_0$ is given by a single right-over-left crossing, which is then undone by applying $T$.

    We turn to the generators of type (b). In light of the presentation $\pair{R,T \mid R^2 = (TR)^3 = id}$ for $\PSL_2(\Z)$ given in \Cref{theorem:psl2z}, $\ker(\phi) \le \Gamma$ is generated by {\em $\Gamma$-conjugates} of the elements $R^2$ and $(TR)^3$. We must verify that any such conjugate $W R^2 W^{-1}$ or $W (TR)^3 W^{-1}$ fixes $\tau_0$, for $W \in \Gamma$ arbitrary.

    \begin{figure}[htbp]
\centering
\labellist
\pinlabel $TR$ at 133 228
\pinlabel $TR$ at 299 228
\pinlabel $TR$ at 467 228
\pinlabel $\sim$ at 150 55
\pinlabel $\sim$ at 305 55
\tiny
\pinlabel $b$ at 50 218
\pinlabel \rotatebox[origin=c]{90}{$b$} at 233 216
\pinlabel $q$ at 397 245
\pinlabel \rotatebox[origin=c]{270}{$b$} at 530 228
\pinlabel $b$ at 223 65
\pinlabel $b$ at 389 65
\endlabellist
\includegraphics[scale = 0.7]{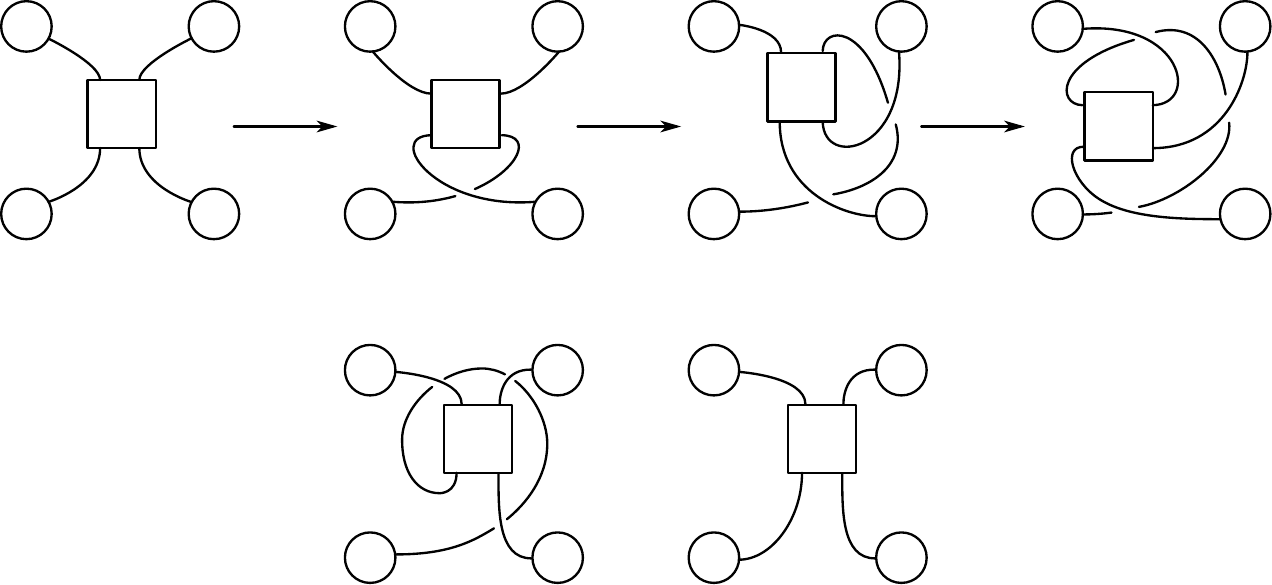}
\caption{Verifying that $(TR)^3$ acts trivially on tangles. }
\label{figure:trcubed}
\end{figure}
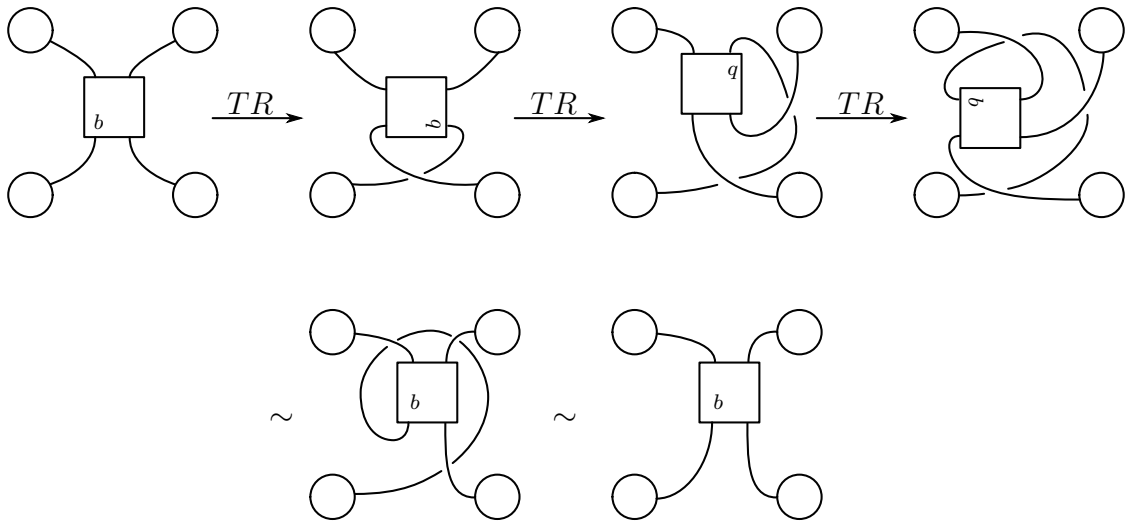

    To do this, observe that $W^{-1}$ takes $\tau_0$ to some arbitrary rational tangle, and that subsequent application of $W$ transforms this back into $\tau_0$. It therefore suffices to show that the elements $R^2$ and $(TR)^3$ act trivially on {\em arbitrary} rational tangles. The triviality of $R^2$ is an immediate consequence of $bdpq$ symmetry of rational tangles as established in \Cref{lemma:symmetric}. Triviality of $(TR)^3$ is verified with a picture calculation as shown in \Cref{figure:trcubed}.\qed

\begin{remark}\label{remark:matricesact}
    Observe that the final stage of this argument shows something rather remarkable: Group Action 1 {\em also} factors through $\phi: \Gamma \to \PSL_2(\Z)$. That is, {\em matrices act on tangles!}
\end{remark}

\section{Why should matrices act on tangles?}
The thing I find so charming about the tangle trick is the way in which it exploits a miraculous-looking connection between the worlds of topology (tangles) and algebra (fractions). In this section, I will offer a structural explanation for this connection. Ultimately I hope to convince you that the existence of the tangle trick is a {\em beautiful inevitability}.

\begin{wrapfigure}{r}{0.4\textwidth}
\centering
\includegraphics[width = 0.38\textwidth]{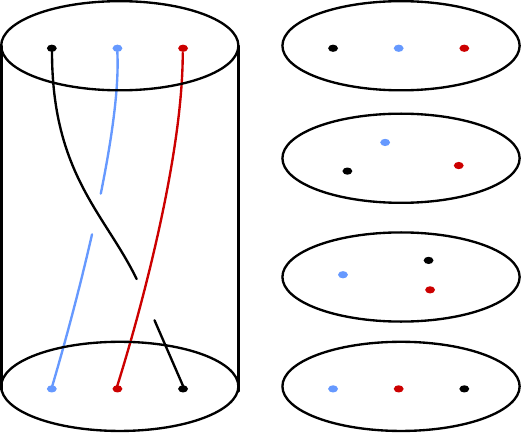}
\end{wrapfigure}

\subsection{Enter braids} To begin with, we will consider a much more natural-looking group action on the set of tangles, by the {\em braid group}.

\begin{definition}[Braid group]
    A {\em braid} on $n$ strands is a collection of $n$ continuous functions $f_i(t): [0,1] \to \C$ with the property that for all $t \in [0,1]$, the points $f_1(t), \dots, f_n(t)$ are pairwise distinct. We require that the {\em set} of beginning ($t = 0$) and end ($t = 1$) points coincide, but not necessarily that $f_i(0) = f_i(1)$ for all $i$. Taking the ``trace'' of the $f_i$'s by the composite maps $t \mapsto (f_i(t),t)$ reveals a picture of $n$ non-crossing strands properly embedded in $\C \times [0,1]$. We consider braids up to {\em isotopy}; informally this means that, like with tangles, we think of braids as being made of physical string, and are allowed to deform so long as the strands never intersect. Technically, an isotopy of a braid is a set of one-parameter continuous deformations $f_{i,s}(t)$ for $s \in [0,1]$, such that for each fixed $s$, the points $f_{i,s}(t)$ remain pairwise-disjoint. We also require $f_{i,s}(0) = f_i(0)$ and $f_{i,s}(1) = f_i(1)$ for all $s$, so that endpoints of the braids remain fixed under isotopy.

    The set of braids up to isotopy forms a group under concatenation: simply {\em stack} two braids on top of each other (and reparameterize $t$). Inversion is given by vertical reflection of the braid, or equivalently, by reversing the flow of $t$.

    The keen-eyed reader will spot that the definition we have given admits a more structural incarnation: we are describing the {\em fundamental group} of the {\em configuration space} $\Conf_n(\C)$ consisting of all collections of $n$ distinct points in $\C$: 
    \[
    B_n = \pi_1(\Conf_n(\C)). 
    \]
\end{definition}

The definition we have given makes an action of $B_3$ on the set $\cT$ of rational tangles relatively apparent.

\begin{wrapfigure}{l}{0.40\textwidth}
\centering
\labellist
\pinlabel $=$ at 180 60
\small
\pinlabel \rotatebox[origin=c]{315}{$b$} at 85 65
\pinlabel \rotatebox[origin=c]{315}{$b$} at 220 83
\endlabellist
\includegraphics[width=0.38\textwidth]{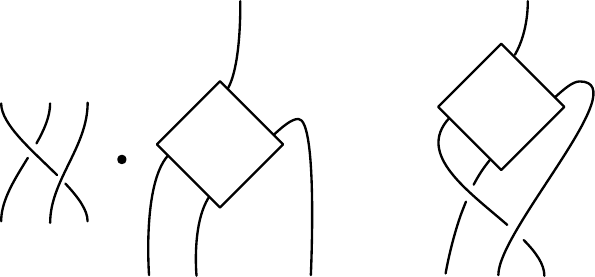}
\caption{The braid group action on tangles.}
\label{figure:braidaction}
\end{wrapfigure}

\para{Group action 3: braiding tangles} Let $\tau$ be a rational tangle, and let $\beta$ be a braid on three strands. Then there is an action $\beta \cdot \tau$ defined as follows. Imagine suspending $\tau$ from the ceiling, hanging from the rope end at position $4$ on the square. This leaves three ends dangling freely; to let $\beta$ act on $\tau$, simply braid these ends as described by $\beta$.

While it is clear that $\beta \cdot \tau$ is some kind of a ``tangle'', it is somewhat less clear that $\beta \cdot \tau$ is still {\em rational} if $\tau$ was. To see this, we will appeal to the following basic fact about braid groups.

\begin{figure}[htbp]
\centering
\labellist
\small
\pinlabel $\sigma_1$ at 40 -5
\pinlabel $\sigma_{n-1}$ at 200 -5
\endlabellist
\includegraphics[scale=1]{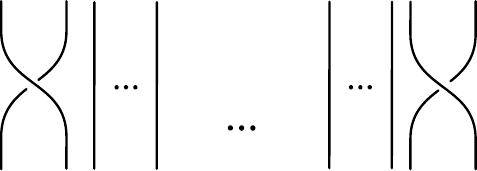}
\caption{The generators $\sigma_1, \dots, \sigma_{n-1}$ of $B_n$}
\label{figure:braidgens}
\end{figure}
\begin{lemma}\label{lemma:braidgens}
    The braid group $B_n$ is generated by the elements $\sigma_1, \dots, \sigma_{n-1}$ shown in \Cref{figure:braidgens}.
\end{lemma}
\begin{proof}
    Every braid (is isotopic to one that) consists of a finite sequence of crossings of one strand over another, and each such crossing is one of the elements $\sigma_i$ or its inverse. 
\end{proof}

\begin{lemma}
    Let $\beta \in B_3$ be given, and let $\tau \in \cT$ be a rational tangle. Under Group Action 3, $\beta \cdot \tau$ is likewise a rational tangle.
\end{lemma}
\begin{proof}
    Following \Cref{lemma:braidgens}, to see that the braid action preserves rationality of tangles, it suffices to verify this for the two generators $\sigma_1$ and $\sigma_2$ of $B_3$. As can be readily inferred from \Cref{figure:braidaction}, the action of $\sigma_1$ coincides with $T$, and the action of $\sigma_2$ coincides with $RTR$.
\end{proof}

The action of $B_3$ on $\cT$ is natural, but the connection with $\PSL_2(\Z)$ is not yet clear. As a first hint, we saw above that $\sigma_1$ acts as $T$ on $\cT$. With a little more work, we can also identify a braid that acts as $R$.

\begin{lemma}\label{lemma:R}
    Let $\tau \in \cT$ be a rational tangle. Then the rational tangle $R \cdot \tau$ (under Group Action 1) is given by $(\sigma_1 \sigma_2 \sigma_1) \cdot \tau$ under Group Action 3.
\end{lemma}
\begin{proof}
    See \Cref{figure:raction}.
\end{proof}
\begin{figure}[htbp]
\centering
\labellist
\pinlabel $R$ at 135 220
\pinlabel $\sigma_1\sigma_2\sigma_1$ at 20 80
\pinlabel $\mbox{rotate}+bdpq$ at 285 80
\tiny
\pinlabel \rotatebox[origin=c]{315}{b} at 14 229
\pinlabel \rotatebox[origin=c]{45}{b} at 247 210
\pinlabel \rotatebox[origin=c]{315}{b} at 109 80
\endlabellist
\includegraphics[scale=0.7]{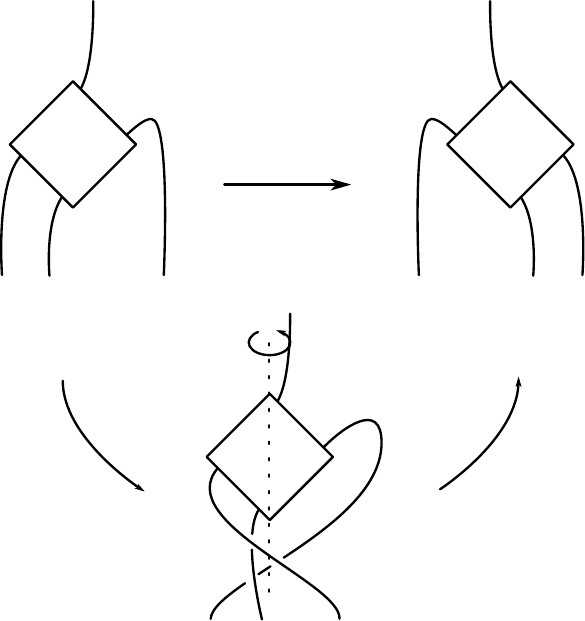}
\caption{Expressing $R$ via the action of $\sigma_1\sigma_2\sigma_1$.}
\label{figure:raction}
\end{figure}
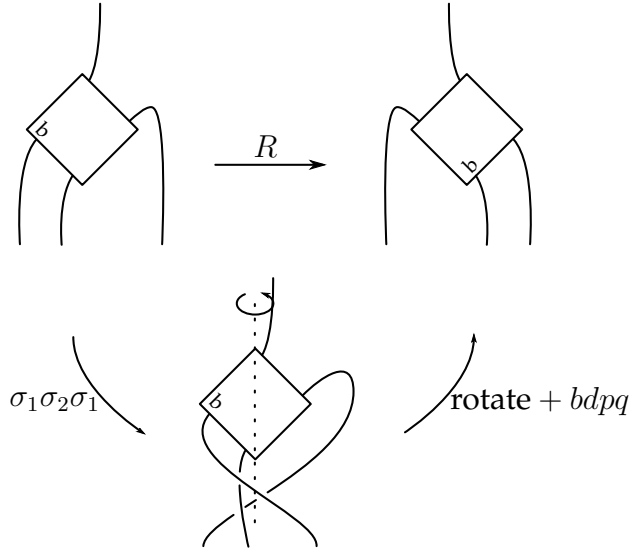

Recalling from the proof of \Cref{untangle0} that $R^2$ acts trivially on $\cT$, it follows that the element $(\sigma_1 \sigma_2 \sigma_1)^2 \in B_3$ likewise acts trivially. This fact ends up being a crucial bridge between braids and matrices!

\begin{proposition}\label{prop:b3modz}
    The center $Z(B_3)$ of $B_3$ is infinite cyclic and is generated by the element $(\sigma_1 \sigma_2 \sigma_1)^2 = (\sigma_1 \sigma_2)^3$. Moreover, there is an isomorphism
    \[
    B_3 / Z(B_3) \cong \PSL_2(\Z)
    \]
    with
    \[
    \sigma_1 \mapsto t = \begin{pmatrix}
        1 & 1\\ 0 & 1
    \end{pmatrix}, \qquad \sigma_2 \mapsto rtr = \begin{pmatrix}
        1 & 0 \\ -1 & 1
    \end{pmatrix}.
    \]
\end{proposition}
\begin{proof}
    These facts are not especially hard to establish via pure group theory, at least the more restrictive assertion (sufficient for our purposes) that the subgroup generated by $(\sigma_1 \sigma_2 \sigma_1)^2$ is central (hence normal) and has quotient $\PSL_2(\Z)$. One proceeds via the famous presentation
    \[
    B_3 = \pair{\sigma_1, \sigma_2 \mid \sigma_1 \sigma_2 \sigma_1 = \sigma_2 \sigma_1 \sigma_2}
    \]
    (the above relation is known as the {\em braid relation}), seeing that the induced presentation for $B_3/\pair{(\sigma_1 \sigma_2 \sigma_1)^2}$ can be transformed into the presentation for $\PSL_2(\Z)$ obtained above. It is somewhat more subtle to show that $(\sigma_1 \sigma_2 \sigma_1)^2$ {\em generates} the center. This is equivalent to the fact (easily checkable by calculation) that $\PSL_2(\Z)$ has trivial center. One can also appeal to the theory of mapping class groups - see \cite[Section 9.2]{FM}.
\end{proof}

\begin{corollary}\label{corollary:factors}
    Group Action 3 of $B_3$ on $\cT$ factors through the quotient $B_3 / Z(B_3) \cong \PSL_2(\Z)$.
\end{corollary}

The upshot of this is that we are edging towards a satisfactory understanding of the mechanism underlying the tangle trick. We start with a natural action of the braid group on the set of rational tangles. \Cref{prop:b3modz} and \Cref{corollary:factors} together imply that this action factors through a quotient of $B_3$ which {\em happens} to be $\PSL_2(\Z)$, and from here we exploit the faithful action of $\PSL_2(\Z)$ on $\Q\cup \{\infty\}$ to track the action. 

The remaining mystery is why the isomorphism 
\[
B_3 / Z(B_3) \cong \PSL_2(\Z)
\]
should exist. What is a coherent explanation for this connection between braids and matrices?

\subsection{The mystery revealed: moduli spaces}\label{section:EC} The answer to this question will take us to the interface of topology and algebraic geometry. As a disclaimer, this section is meant as a high-level overview. I ask the forgiveness of the experts for the many subtleties I will attempt to finesse for the sake of clarity and concision, and I ask the reader for understanding that I will assume more mathematical sophistication and that most details will be suppressed. 

To start with, a bit of topological philosophy is in order. One of the basic principles of geometric group theory (which I am considering here to be a branch of topology, for better or for worse) is that it is profitable to understand {\em groups} as {\em fundamental groups of spaces}. Since the fundamental group is a {\em functor}, one coherent structural mechanism for understanding group homomorphisms is to understand continuous maps between associated spaces. 

In our case, this means the following: we seek to find spaces $X$ and $Y$ with fundamental group $\pi_1(X) = B_3$ and $\pi_1(Y) = \PSL_2(\Z)$, and some map $f: X \to Y$ which induces the quotient $B_3 \to \PSL_2(\Z)$ on fundamental groups. We have already found a good candidate for $X$, in the space $\Conf_n(\C)$. It remains to find a space $Y$ with fundamental group $\PSL_2(\Z)$.

For this, we encounter the notion of a {\em moduli space}. Coarsely,\footnote{For the {\em cognoscenti}: the pun is very much intended.} a moduli space is just a topological space whose points {\em parameterize} isomorphism classes of some kind of mathematical object. For us, the relevant mathematical object will be an {\em elliptic curve}. 

There are two points of view on elliptic curves that we will encounter. The first, which we shall call the {\em geometric} perspective, is that an elliptic curve is a quotient $E = \C/\Lambda$, where $\Lambda \le \C$ is a {\em lattice}, i.e. a subgroup isomorphic to $\Z^2$ generated by two elements of $\C$ that are linearly independent over $\R$. The second, {\em algebraic} point of view, is that an elliptic curve is the solution set in $\C^2$ to an equation of the form $y^2 = f(x)$, where $f(x)$ is a monic polynomial of degree $3$ with distinct roots. 

It is far from clear that these two points of view describe the same set of objects. Nevertheless, since work of Weierstrass, Abel, and Jacobi in the first half of the 19th century, it has been understood that each geometric elliptic curve $\C/\Lambda$ corresponds to an algebraic one $y^2 = f(x)$ and vice versa.\footnote{\label{aj}Unfortunately it would expand the scope of the article considerably to do these beautiful ideas justice, and some buzzwords will have to suffice. Weierstrass constructed his ``$\wp$ function'' on a geometric elliptic curve $\C/\Lambda$ and showed that it satisfied a nonlinear first-order ODE of the form $(\wp')^2 = f(\wp)$ for $f$ a cubic polynomial. Conversely, given an algebraic elliptic curve, a geometric elliptic curve $\C/\Lambda$ emerges as its {\em Jacobian}, by integrating a canonical differential form present on the algebraic model.} {\bf Ultimately it is this correspondence that explains the tangle trick!}\\

To elucidate this, let us return to the notion of a {\em moduli space}. It is not so hard to make sense of the {\em space of all lattices}. To construct this, we start with the notion of a {\em marked lattice}, which is nothing more than a pair $(z,w)$ of complex numbers linearly independent over $\R$; taking the $\Z$-span of $z,w$ then yields a lattice. A {\em marking} of a lattice $\Lambda$ is a choice of elements $z,w \in \Lambda$ for which $\Lambda = \pair{z,w}$. For reasons that we don't quite have the space to get into, it is best to consider (marked) lattices up to {\em homothety}, and take equivalence classes under the action of $\C^*$ (via $\lambda \cdot (z,w) = (\lambda z, \lambda w)$). After homothety, we can take $z = 1$ and $w = \tau$ a complex number in the {\em upper half-plane} $\hyp$ of complex numbers with positive imaginary part. The conclusion is that the space of {\em marked} lattices is nothing but $\hyp$.

To consider the space of {\em unmarked} lattices, we consider the set of markings of a given lattice $\Lambda$. A little bit of basic algebra shows that if $(z,w)$ and $(z',w')$ are markings of $\Lambda$, {\em then there is an element $A \in \SL_2(\Z)$ such that $A (z,w) = (z',w')$} (the product here denoting ordinary matrix multiplication). The upshot is that one can construct the space of unmarked lattices by first constructing the space $\hyp$ of marked lattices, then identifying two points in $\hyp$ if they give the same unmarked lattice. Thus, the {\em moduli space of elliptic curves} $\cM_{1,1}$ is constructed as the quotient 
\[
\cM_{1,1} = \hyp/\SL_2(\Z),
\]
where $\SL_2(\Z)$ acts on $\hyp$ by changing the marking. Given our normalization conventions (insisting that $z = 1$), the {\em linear} action of $\SL_2(\Z)$ on the space of markings $(z,w)$ descends to an action of $\SL_2(\Z)$ on $\hyp$ {\em via M\"obius transformations},\footnote{This is not entirely obvious, but can be easily verified directly.} and as before, this in fact descends to an action of $\PSL_2(\Z)$. 

Recall that our immediate goal was to identify a space $Y$ with fundamental group $\PSL_2(\Z)$. Let us see the extent to which the discussion above achieves this. In general, topology tells us that if $S$ is a simply-connected ($\pi_1(S)$ trivial) topological space, and $G$ is a group acting on $S$ {\em freely} (nontrivial elements act without fixed points) and ``properly discontinuously'' (a technical condition I won't discuss further), then the quotient space $S/G$ has fundamental group $G$. This is {\em almost} true in the case of $\PSL_2(\Z)$ acting on $\hyp$: the mysterious proper discontinuity condition is satisfied, but the action is not quite free: for instance, the M\"obius action of $r = \begin{pmatrix}
    0&-1\\ 1 & 0
\end{pmatrix}$ fixes $i = -1/i \in \hyp$. 

The end result is that it is not {\em literally} true that the quotient $\hyp/\PSL_2(\Z)$ has fundamental group $\PSL_2(\Z)$ (in fact, the quotient space is simply-connected). But it turns out to be {\em very nearly} true. The workaround is to {\em enhance} the quotient $\hyp/\PSL_2(\Z)$ with extra information at the fixed points of the action and remember the stabilizer subgroups. This leads to the notion of an {\em orbifold} (or, for the algebro-geometrically inclined, a {\em stack}). I will not even pretend to get into a discussion of these concepts in any depth, but suffice it to say that there is a generalization of the notion of fundamental groups for orbifolds, and that the {\em orbifold fundamental group} is what we expect:
\[
\pi_1^{orb}(\hyp / \PSL_2(\Z)) \cong \PSL_2(\Z).
\]

It is time for the final act of our magic show. Our objective in this section has been to identify a map between spaces $f: X \to Y$ that gives a topological incarnation of the mysterious homomorphism $B_3 \to \PSL_2(\Z)$ that makes the tangle trick work. My final assertion is that there is an extremely natural map
\[
AJ: \Conf_3(\C) \to \hyp/\PSL_2(\Z)
\]
which induces the map $B_3 \to \PSL_2(\Z)$ upon passing to (orbifold) fundamental groups. The notation $AJ$ makes reference to the {\em Abel-Jacobi map}, of which this is essentially a special case. 

The basic idea of $AJ$ is to convert an {\em algebraic} elliptic curve (i.e. the solutions to some equation $y^2 = f(x)$) into its corresponding {\em geometric} description as some $\C/\Lambda$ (see \Cref{aj} above). But the domain of $AJ$ is the space $\Conf_3(\C)$ - how do we associate points in this space to such equations? It is actually quite simple: a point in $\Conf_3(\C)$ is by definition an unordered set $\{z_1, z_2, z_3\}$ of distinct complex numbers. By the Fundamental Theorem of Algebra, such tuples are in bijective correspondence with the set of monic cubic polynomials with distinct {\em roots}. Explicitly, the association assigns the tuple $\{z_1, z_2,z_3\} \in \Conf_3(\C)$ to the elliptic curve
\[
y^2 = (z-z_1)(z-z_2)(z-z_3).
\]
From here, the foundational 19th century results alluded to above convert the given equation into a geometric elliptic curve of the form $\C/\Lambda$.

\section{A final curiosity: why no untwisting?}

Let us return to a somewhat subtle aspect of the tangle trick. Recall that when the {\sc magician} returns to the room and is told the tangle invariant by the {\sc assistant}, they then proceed to (apparently) {\em further complicate} the tangle, adding in additional twists and turns in the same direction as before. Following the explication of the solution process (\Cref{example}), we are at least assured that it is {\em possible} to untangle any rational tangle by some further sequence of twists and turns. But I'd like to draw attention to how weird this is: in most actions of infinite groups on spaces (e.g. $\Z^2$ acting by translation on $\R^2$, or the free group $F_2$ acting on an infinite $4$-valent tree), if you keep acting only by {\em positive} generators, you never return to where you start. (Admittedly, this can be artificially circumvented by adding in new generators equal to the inverses of others, but this is not really in the spirit of what I'm getting at). What is it about the action of $B_3$ on rational tangles that allows you to act transitively using only the {\em monoid} of positive words?

To answer this, we return to the fact mentioned in \Cref{prop:b3modz} that $Z(B_3)$ is generated by the element $(\sigma_1 \sigma_2 \sigma_1)^2$. Let us define
\[
\Delta = \sigma_1 \sigma_2 \sigma_1 = \sigma_2 \sigma_1 \sigma_2,
\]
so that $Z(B_3)$ is generated by $\Delta^2$. Following Garside \cite{garside}, the element $\Delta$ is called the {\em fundamental element} of $B_3$.

The fact that positive moves suffice for untangling is explained by properties of $\Delta$. As we observed after \Cref{lemma:R}, $\Delta^2$ acts trivially on $\cT$. Combined with the fact that $\Delta^2$ is central, this implies the following explanation of our phenomenon.

\begin{lemma}
    Let $W$ be a word in the letters $R$ and $T$ {\em and their inverses}. Then there is a word $W^+$ on $R$ and $T$ only, such that for any rational tangle $\tau$,
    \[
    W \cdot \tau = W^+ \cdot \tau.
    \]
\end{lemma}
\begin{proof}
    We proceed by induction on the length $n$ of $W$, taking the empty word as vacuous base case. To complete the inductive step, it suffices to show that if $W^+$ is any positive word, then the action of the words $\sigma_1^{-1}W^+$ and $\sigma_2^{-1}W^+$ on $\cT$ coincide with the action of positive words.

    Consider a tangle $\sigma_1^{-1} W^+ \cdot \tau$. As $\Delta^2$ acts trivially on $\tau$ and is central,
    \[
\sigma_1^{-1} W^+ \cdot \tau = \sigma_1^{-1} W^+ \Delta^2 \cdot \tau = \sigma_1^{-1}\Delta^2 W^+ \cdot \tau.
    \]
    Now note that the word $\sigma_1^{-1} \Delta^2 = \sigma_1^{-1} (\sigma_1 \sigma_2 \sigma_1)^2 = \sigma_2 \sigma_1^2 \sigma_2 \sigma_1$ is positive. The same technique shows that $\sigma_2^{-1}W^+$ can be replaced with the word $\sigma_2^{-1} \Delta^2 W^+$. Via the braid relation $\sigma_1 \sigma_2 \sigma_1 = \sigma_2 \sigma_1 \sigma_2$, we can write $\Delta^2 = (\sigma_2 \sigma_1 \sigma_2)^2$, and as before, this exhibits $\sigma_2^{-1} \Delta^2$ as a positive word of length five. 
\end{proof}

To summarize, the salient properties of $\Delta$ were that (1) $\Delta^2$ is central, (2) $\Delta^2$ can be expressed as a positive word starting with any positive generator, and (3) $\Delta^2$ is in the kernel of the action of $B_3$ on $\cT$. The first two of these conditions are intrinsic to $B_3$, and were studied in Garside's paper \cite{garside} referenced above. Garside showed that every braid group $B_n$ has an element $\Delta$ with these properties, and used this to give a solution to the word problem in $B_n$. Subsequently this was generalized by Deligne \cite{deligne} and Brieskorn-Saito \cite{BS}, who considered a more general class of so-called {\em Artin groups}, a special class of which (those of {\em finite type}) were found to be in possession of analogous elements.

\para{Acknowledgements} I'd like to thank Tim Black for introducing me to the tangle trick and for many companionable hours spent below ground level in Hyde Park. Thanks are due to Corentin Lunel and Lionel Lang for pointing out an error in \Cref{example}. Thanks also to Curt McMullen for sharing his own recollection of the trick. Lastly, I'd like to thank John H. Conway, whom I never had the pleasure of meeting, but whose mathematics I have long admired.  

    \bibliography{references}{}
	\bibliographystyle{alpha}

\end{document}